\documentclass[preprint, 1p,number,12pt]{elsarticle}

\usepackage{amsfonts, amssymb,tikz,pictex, latexsym, hyperref, amsthm,color,mathrsfs,graphicx, amsmath,caption,tabularx,verbatim,comment,pdfpages}

\numberwithin{equation}{section}
\makeatletter
\renewcommand{\subsection}{\@startsection
	{subsection}{2}{0mm}{\baselineskip}{-0.25cm}
	{\normalfont\normalsize\bf}}
\makeatother

\newtheorem{thm}{Theorem}[section]
\newtheorem{theorem}[thm]{Theorem}

\newtheorem{lemma}[thm]{Lemma}

\newtheorem{corollary}[thm]{Corollary}

\newtheorem{proposition}[thm]{Proposition}

\numberwithin{equation}{section}
\theoremstyle{remark}
\newtheorem{remark}[thm]{Remark}

\theoremstyle{definition}

\newcommand{\cH}{{\mathcal H}}

\newcommand{\cC}{{\mathcal C}}

\newcommand{\cD}{\mathcal{D}}

\newcommand{\fq}{{\mathbb{F}_q}}
\newcommand{\fqs}{{\mathbb{F}_{q^2}}}

\newcommand{\Aut}{\mathrm{Aut}}

\newcommand{\bfq}{{\mathbb F_q}}

\makeatletter
\@ifpackageloaded{hyperref}%
{\newcommand{\mylabel}[2]
	{\protected@write\@auxout{}{\string\newlabel{#1}{{#2}{\thepage}%
				{\@currentlabelname}{\@currentHref}{}}}}}%
{\newcommand{\mylabel}[2]
	{\protected@write\@auxout{}{\string\newlabel{#1}{{#2}{\thepage}}}}}
\makeatother

\definecolor{Amaranto}{rgb}{0.9, 0.17, 0.31}
\definecolor{Borgogna}{rgb}{0.5, 0.0, 0.13}

\begin{document}
	\title{On permutation quadrinomials from Niho exponents in characteristic two}
	
	\author[1]{Vincenzo Pallozzi Lavorante\corref{cor1}%
		}
	\ead{vincenzop@usf.edu}
	\cortext[cor1]{Corresponding author}
	
	\affiliation[1]{organization={Department of Mathematics and Statistics}, addressline={University of South Florida}, postcode={FL 33620}, city={Tampa}, country={USA}}

	\begin{abstract}
		Recently Zheng et al. \cite{zheng2022more} characterized the coefficients of $ f(x) =
		x + a_1x^{s_1(2^m-1)+1} + a_2x^{s_2(2^m-1)+1} + a_3x^{s_3(2^m-1)+1} $ over $\mathbb{ F}_{2^{2m}}$ that lead $ f(x) $ to be a permutation of $\mathbb{ F}_{2^{2m}}$ for $ (s_1, s_2, s_3) = (\frac{1}{4},1, \frac{3}{4})$. They left open the question whether those conditions were also necessary. In this paper, we give a positive answer to that question, solving their conjecture.
	\end{abstract}
	
	\maketitle

	\begin{small}
		
		
		
	\end{small}

	\section{Introduction}
	Let $q=p^m$ be a prime power. Let $\bfq$ denote the finite fields of $q$ elements. A polynomial $f \in \bfq[x]$ is called a permutation polynomial (PP) of $\bfq$ if the induced mapping $x \mapsto f(x)$ is a bijection of $\bfq$. Several authors in recent history have focused on permutation polynomials and their applications.
	For example, PPs have been widely used in coding theory and cryptography, and we refer the reader to \cite{Hou_Survey,Hou-CM-2015} for a survey on the latest advances.
	Recently, PPs taking simple forms and few terms have attracted much interest and have been deeply investigated.
	In \cite{Zieve_mq+1} the authors provided a powerful method to construct PPs using the set of $q+1$-th roots of unity and it was generalized in \cite{Wang-LNCS-2007,Zieve-arXiv1310.0776}
	Along this view, several new families of permutation polynomials have been constructed and we refer the reader to \cite{Hou-arXiv:1609.03662, Lappano-pc,Lappano-thesis, hou2023new,hou2022general} for more details.
	
	Another way to look at the PPs is based on algebraic curves over finite fields. In \cite{Hou-FFA-2018}, it was shown how to use the theory of algebraic curves to determine whether a polynomial is a permutation polynomial or not. 
	
	Permutation trinomials with Niho exponents of the form $ f(x) =
	x + a_1x^{s_1(2^m-1)+1} + a_2x^{s_2(2^m-1)+1}  \in \fqs[x]$, have attracted much interest in recent years. See for example \cite{Bartoli_HWPP, Hou-FFA-2015b, hirschfeld1979projective}. The parameters $ s_1, s_2 $ should be read modulo $ q+1 $. 
	Given $ (s_1, s_2), $ finding conditions on $ a_1, a_2 $ that are sufficient and necessary for $ f $ to be
	a permutation polynomial of $\fqs$ is a hard question and some progress have been done in that direction. See \cite{hou2014determination, Hou-FFA-2015b, bartoli2018conjecture}.
	
	However, the situation for permutation quadrinomials is different. Let $  f(x) = x + a_1x^{s_1(2^m-1)}+ a_2x^{s_2(2^m-1)}+ a_3x^{s_3(2^m-1)} $. Recently, Tu et al. investigated the case of $ (s_1, s_2, s_3) = (-1, 1, 2) $ under some restrictive conditions \cite{tu2021binomial}.
	In \cite{zheng2022more} the authors provided more classes of permutation quadrinomials from Niho exponents in characteristic two for $(s_1,s_2,s_3)=(\frac{-1}{2^k-1},1,\frac{2^k}{2^k-1})$, $(s_1,s_2,s_3)=(\frac{1}{2^k+1},1,\frac{2^k}{2^k+1})$, where $m$ and $k$ are positive integers and $ (s_1, s_2, s_3) = (\frac{1}{4},1, \frac{3}{4})$.
	Their work focuses on finding sufficient conditions for the polynomials to be permutation polynomials and it is based on arithmetic over finite fields. What is often harder is to show whether those conditions are also necessary or not.
	In fact, they also suggested that the conditions given were necessary for $m$ big enough, but they have not found efficient techniques to show those facts.
	
	Very recently the problem of characterizing  permutation quadrinomials was also address by Ding and Zieve in \cite{ding2022determination}, where the authors determined a very large class of permutation quadrinomials by using novel geometric techniques (even when the Weil bounds do not provide useful information). In particular they were able to solve two out of the three conjectures presented in \cite[Th 1.1 and 1.3]{zheng2022more}.
	
	In this paper, we aim to answer the conjecture left open, that is investigating whether the conditions in in \cite[Theorem 1.4]{zheng2022more} are also necessary. We will use the Hasse-Weil type theorems to prove necessary conditions for a polynomial to be a permutation polynomial.
	In particular we will give a complete answer to the question, see Theorem \ref{th:main2}.

	\section{Setting and known results}
	Let $ q=2^m $ be a prime power and $\fqs $ be the finite field of $ q^2 $ elements.
	Let $ a_1, a_2, a_3 \in \fqs $ and denote $ \theta_1=1+a_1^{q+1}+a_2^{q+1}+a_3^{q+1} $, $ \theta_2=a_1^q+a_3 a_2^q $, $ \theta_3=a_3+a_2 a_1^q $, $\theta_4 = a_1^{q+1}+a_3^{q+1}$ and $ \theta_4'=\theta_1+\theta_4=1+a_2^{q+1} $. Note that
	\[\theta_2^{q+1}+\theta_3^{q+1}=\theta_4 \theta_4 '.\]

	We now summarize the previous results we need in this chapter. See \cite{zheng2022more} for more details.
	
	\begin{theorem}[{\cite[Theorem 1.4]{zheng2022more}}]\label{th:1.4}
		Let $n=2m$ be a positive integer. Let $(s_1,s_2,s_3)=(\frac{1}{4},1,\frac{3}{4})$ and $a_1,a_2,a_3 \in \fqs$. Then $f(\mathtt{X})=\mathtt{X}+a_1 \mathtt{X}^{s_1(q-1)+1}+a_2\mathtt{X}^{s_2(q-1)+1}+a_3\mathtt{X}^{s_3(q-1)+1}$ is a PP of $\fqs$ if either
		\begin{equation}\label{eq:1}
			\theta_4 \ne 0, \quad 	\theta_2=0 \quad \mbox{ and } \quad a_3 \in \mu_{q+1},\quad  a_3 \notin \{x^3 | x \in \mu_{q+1}\}	
		\end{equation}
		or 
		\begin{equation}\label{eq:2}
			\theta_1 \ne 0,	\theta_2 \ne 0, \theta_4=0, \theta_3=\theta_2^{2q-1}  \mbox{ and }   x^3+x+\frac{\theta_1^2}{\theta_2^{q+1}}=0 \mbox{ has no solutions over } \fq.
		\end{equation}
	\end{theorem}
	
	The aim of this paper is to answer the question left open by the authors in \cite[Theorem 1.4]{zheng2022more} and prove that conditions \eqref{eq:1} and \eqref{eq:2} are also necessary.
	
	The main result is stated in the following theorem.
	
	\begin{theorem}\label{th:main2}
		Let $m \geq 9$ be an integer and $q=2^m$. With the notation above, if the polynomial
		\begin{equation}\label{eq:f(x)}
			f(x)=x+a_1x^{s_1(q-1)+1}+a_2x^{s_2(q-1)+1}+a_3x^{s_3(q-1)+1}
		\end{equation} is a PP of $\fqs$ then 
		\begin{itemize}
			\item if $\theta_2=0$ then $\theta_4 \ne 0 $, $ a_3 \in \mu_{q+1} $, $a_3 \notin \{x^3 | x \in \mu_{q+1}\}$;
			\item if $\theta_2 \ne0$, then $\theta_4=0$, $\theta_1 \ne 0$,  $\theta_3=\theta_2^{2q-1}$  and  \begin{equation}\label{eq:trin}
				x^3+x+\frac{\theta_1^2}{\theta_2^{q+1}}=0 
			\end{equation} has no solutions in $ \fq $.
		\end{itemize}
	\end{theorem}
	
	\begin{corollary}
	{\em{Conditions} \eqref{eq:1}} and	{\eqref{eq:2}} of {\em{Theorem \ref{th:1.4}}} are also necessary.
	\end{corollary}

	\section{Algebraic curves and necessary conditions}
	It is well known that polynomials of the form $f(x)=x h(x^{q-1})$ permutes $\fqs$ if and only if $g(x)=xh(x)^{q-1}$ permutes the set $\mu_{q+1}$ of the $(q+1)$-roots of unity in $\fqs$. See \cite[Theorem 1.2]{Zieve_mq+1}.
	For $f(x)$ in Equation \eqref{eq:f(x)} this means to prove that the rational function
	\[p(x)=\frac{x^4+a_1^{q}x^3+a_3^qx+a_2^q}{a_2x^4+a_3x^3+a_1x+1}\]
	permutes $\mu_{q+1}$, see \cite[Section 5]{zheng2022more} for more details.
	
	Let $\cC$ be the plane curve associated to $p(x)$, with equation $F(X,Y)=(p(X)-p(Y))/(X-Y)=0$, that is
	
	\begin{equation}\label{eq:CurveC}
		\begin{aligned}
			F(X,Y)=&\frac{\left(a_1 Y+{a_2} Y^4+{a_3} Y^3+1\right) \left(X^3 {a_1}^q+{a_2}^q+{a_3}^q X+X^4\right)}{X+Y}+ \\
			&	+ \frac{\left({a_1} X+{a_2} X^4+{a_3} X^3+1\right) (Y^3 {a_1}^q+{a_2}^q+Y {a_3}^q+Y^4 )}{X+Y}=0,
		\end{aligned}
	\end{equation} 
	or equivalently
	\begin{equation}\label{eq:C}
		\begin{aligned}
			F(X,Y)=&	\theta_3^q + \theta_3 X^3 Y^3 + 
			\theta_4 X Y (X + Y) +  \theta_4'(X + Y)^3 +\\
			&+\theta_2 (X Y  + (X + Y)^2) + \theta_2^q (X^2 Y^2  + X Y (X + Y)^2)=0.
		\end{aligned}
	\end{equation}
	$\cC$ is a curve defined over $ \fqs $ and $p(x)$ permutes $\mu_{q+1}$ if and only if there are no points $(X,Y) \in \cC \cap \mu_{q+1}^2$ such that $X \ne Y $.
	Since understanding whether or not $\cC$ has points in the set $\mu_{q+1}$ is not always an easy task to do, we will consider also the following idea.
	Choose an element $e \in \fqs$ such that $e^q=e+1$. Every $ x \in \mu^2_{q+1}$ different from $1$ can be written as $x=\frac{X+e}{X+e+1}$, where $X$ runs over $\fq$. 
	Then $p(x)$ permutes $\mu_{q+1}$ if and only if $p(\phi(x)) \colon \fq \cup \{\infty\} \to \mu_{q+1}$, where $\phi(x)=\frac{x+e}{x+e+1}$, $\phi(\infty)=1$, is a bijection.
	This is equivalent to ask that $H(x)=p(\phi(x))_{|\fq} \colon \fq \to \mu_{q+1} \setminus \{(a_1+a_2+a_3+1)^{q-1}\}$ is a bijection.
	Let $\mathcal{H}$ be the affine curve defined by $(H(X)-H(Y))/(X-Y)=0$.
	It is easily checked that $\cH$ is defined over $\fq$ and therefore $H(x)$ is a bijection if and only if $\cH$ has no $\fq$-rational points off the line $X=Y$.

	Moreover, an equation for $\cH$ is given by $H(X,Y)=(X+e+1)^3(Y+e+1)^3F(\phi(X),\phi(Y))$ and $\cH$ is $\fqs$-birationally equivalent to $\cC$: let \[\psi(X,Y)=\Big(\frac{X(e+1)+e}{X+1},\frac{Y(e+1)+e}{Y+1}\Big).\]
	then $ (1+X)^3(1+Y)^3H(\psi(X,Y))=F(X,Y) $ and the two curves are $\fqs$-birationally equivalent.
	
	\begin{proposition}\label{pr:A-P}
		Let $q \geq 512$. If $f(x) \in \fqs[x]$ is a PP then $\cC$ is not absolutely irreducible over $\fqs$.  
	\end{proposition}
	\begin{proof}
		If $\cC$ is absolutely irreducible over $\fqs$ then $\cH$ is absolutely irreducible over $\fq$. Since $ \cH $ has degree at most $6$, the Hasse-Weil bound implies that $\cH$ has at least an affine rational point $(a,b)$ with $a \ne b$ whenever
		\begin{equation}\label{eq:q}
			q+1-20\sqrt{q}-12 \ge 0,
		\end{equation}
		where $12$ is the maximum number of points belonging either to the line $X=Y$ or to the infinity line.
		Equation \eqref{eq:q} is satisfied for every integer greatest than $421$. Thus, if $q=2^m \geq 512$, $\cH$ has an $\fq$-rational point $(a,b)$, with $a \ne b$.
		Consequently, we obtain a point $\Big(\frac{a+e}{a+e+1},\frac{b+e}{b+e+1}\Big)=(a',b') \in \mu_{q+1}^2 $ such that \[ a' \ne b' \quad \mbox{ and } \quad p(a')=p(b'),\] which is in contrast with $f(x)$ being a PP of $\fqs$.
	\end{proof}
	
	Proposition \ref{pr:A-P} allows us to focus on $\cC$ to obtain necessary conditions on $f(x)$. However we will see that proving the absolutely irreducibility of $\cC$ is not always possible. Thus, in some cases, we will exhibit explicitly points belonging to $\cC \cap \mu_{q+1}^2$, off the line $X+Y=0$. 
	
	Understanding whether $\cC$ is reducible or not may be difficult. For this reason, one can ask for a transformation that sends $\cC$ to a lower degree curve easier to study.
	In particular, the group $\mathfrak{G}$ generated by $ (X,Y) \mapsto (Y,X) $ is a subgroup of $\Aut(\cC)$, the automorphism group of $\cC$.
	Furthermore, let $u=X+Y$, $v=X Y$ and $G(u,v)=F(X,Y)$. 
	Let $\cD$ be the curve defined by  $G(u,v) =0$, that is
	\begin{equation}\label{eq:D}
		\cD \colon \theta_3^q + \theta_4' u^3 + \theta_4 u v + \theta_3 v^3 + \theta_2 (u^2 + v) + 
		\theta_2^q v (u^2 + v) =0,
	\end{equation}
	which is the quotient curve $\cC \textfractionsolidus \mathfrak{G}$.
	When convenient, we will study the connection between $\cC$ and $\cD$ to derive information on the irreducibility of $\cC$.
	
	The paper is organized as follows: Sections \ref{sec:4} and \ref{sec:5} are dedicated to the cases $\theta_{2}=0$ and $\theta_4=0$ respectively, strongly using the connection between $\cC$ and $\cD$.
	Section \ref{sec:6} will be devoted to the case $\theta_{2}\ne 0$ and $\theta_4\ne 0$ and we will focus on the factorization of $\cH$ showing that there must always be an absolutely irreducible component defined over $\fq$ in that case.
	
	\section{Case $\theta_2=0$}\label{sec:4}
	We note that if $\theta_{1}=0$ then $\theta_4 \ne 0$ and $a_3 \in \mu_{q+1}$. In fact since $a_1=a_3^q a_2$ and $a_1^q=a_2^q a_3$, then 
	\[\theta_4=a_1^{q+1}+a_3^{q+1}=a_3^{q+1}(1+a_2^{q+1})=a_3^{q+1}\theta_4,\]
	and $\theta_3=a_3\theta_4$ which justifies $\theta_4 \ne 0$ and $a_3\in \mu_{q+1}$, otherwise $f$ is trivially not a PP (see for example \cite[Result 1.2]{hou2023new}).
	When this happens the equation of $\cD$ is exactly
	\[G(u,v)=a_3^q + u^3 + u v + a_3 v^3,\]
	
	The latter equation needs to be studied also when $\theta_1 \ne 0$, this is why we will see both cases together at the end of this section. See Remark \ref{rk:t1=0}.
	
	Let $\theta_1 \ne 0$. In the following remark, we recall some properties from \cite{zheng2022more}.
	\begin{remark}\label{rk:t2=0}
		When $\theta_2=0$ and $\theta_1 \ne 0$ the following hold
		\[\theta_4=a_3^{q+1}\theta_4', \quad \theta_3=a_3 \theta_4' \quad\mbox{ and} \quad \theta_4' \ne 0 \]
		since $\theta_4 \ne 0$ (and $a_3 \ne 0$), otherwise $f$ is trivially not a PP since (see for example \cite[Result 1.2]{hou2023new}).
	\end{remark}
	Therefore $ \cC $ becomes $F(X,Y)=0$ with
	\begin{equation*}\label{eq:Ct20}
		F(X,Y)=	a_3^q + a_3 X^3 Y^3 + 
		a_3^{q+1} X Y (X + Y) +  (X + Y)^3 
	\end{equation*}
	and	$\cD$ becomes $G(u,v)=0$ with
	\[G(u,v)=a_3^q + u^3 + a_3^{q+1} u v + a_3 v^3.\]
	
	\begin{proposition}\label{pr:absirr}
		The curve $ \cD $ defined by Equation \eqref{eq:D} is absolutely irreducible if and only if $a_3 \notin \mu_{q+1}$.
	\end{proposition}
	\begin{proof}
		If $a_3=0$ then $G(u,v)$ is not absolutely irreducible. Let $a_3 \ne 0$.
		Note that every singular point of $\cD$ is a double point. In fact, we have $\partial _{u v}G \ne 0$ and $\partial _{v u}G \ne 0$.
		The system of partial derivatives is 
		\begin{equation*}
			\begin{cases}
				\begin{aligned}
					&u^2 + a_3^{q+1} v=0\\
					&a_3^{q} u +  v^2=0
				\end{aligned}
			\end{cases}
		\end{equation*}
		
		and it implies that a point $P=(u,v)$ is singular if and only if $P=(a_3^{q+2/3}, a_3^{q+1/3})$ and $P \in \cD$ (note that the cubic roots of $a_3$ are not uniquely determined).
		More precisely $ G(P)=0 $ implies that
		\[a_3^q+a_3^{3q+2}=0\]
		which proves that $\cD$ is singular if and only if $a_3^{q+1}=1$.
		Furthermore, since the equation $v^3 =a_3^{3q+1}$ admits $3$ solutions in the algebraic closure $\overline{\mathbb{F}}_q$ of $\fqs$, we have three double points and the cubic is the union of three non concurrent lines.
		This means that $\cD$ is absolutely irreducible if and only if it is non-singular, namely $a_3 \notin \mu_{q+1}$.
	\end{proof}
	\begin{remark}\label{rk:cubic}
		Since $\mathfrak{G}$ is an automorphism group of $\cC$, there is only one situation in which $\cC$ is reducible whereas $\cD$ is not: when $\cC$ is the product of two cubics, which form an orbit under $\mathfrak{G}$. In fact, in that case, $\cD$ is a cubic curve, which may be irreducible. 
	\end{remark}
	
	\begin{proposition}\label{pr:union}
		The curve $\cC$ is the union of two cubic curves only if $a_3 \in \mu_{q+1}$.
	\end{proposition}
	\begin{proof}
		Since the action of $\mathfrak{G}$ is exchanging the $x$ with the $y$, the only possible factorization of $\cC$ is
		\begin{equation}\label{eq:cubics}
			\begin{aligned}
				&(a_{00} + a_{10}X + a_{20}X^2 + 
				a_{30}X^3 + a_{01}Y + a_{11}XY + 
				a_{21}X^2Y + a_{02}Y^2 + 
				a_{12}XY^2 + a_{03}Y^3)\\
				&(a_{00} + a_{01}X + a_{02}X^2 + 
				a_{03}X^3 + a_{10}Y + a_{11}XY + 
				a_{12}X^2Y + a_{20}Y^2 + 
				a_{21}XY^2 + a_{30}Y^3)=0	
			\end{aligned}
		\end{equation} 
		Note that the equation of $\cC$ is 
		\[a_3^q + X^3 + (a_3^{q+1}+1) X^2 Y + 
		(a_3^{q+1} +1)X Y^2 + Y^3 + a_3 X^3 Y^3=0.\]
		Thus, comparing the coefficients, we see that the only possibility for Equation \eqref{eq:cubics} is 
		\[a_{00}^2 + a_{00} a_{30} X^3 + a_{00} a_{30} Y^3 + a_{30}^2 X^3 Y^3=0.\]
		However, this is admissible if and only if $a_3^{q+1}+1=0$, that is $a_3 \in \mu_{q+1}$.
	\end{proof}
	
	\begin{corollary}\label{cr:Cabsirr}
		Let $a_3 \notin \mu_{q+1}$. The curve $\cC$ is absolutely irreducible.
	\end{corollary}
	\begin{proof}
		Proposition \ref{pr:absirr} implies that for $a_3 \notin \mu_{q+1}$ the curve $\cD$ is absolutely irreducible. The proof follows from Remark \ref{rk:cubic} together with Proposition \ref{pr:union}.
	\end{proof}
	We consider now the case when $\cD$ is not absolutely irreducible. In this case we have 
	\[G(u,v)=a_3^q + u^3 + u v + a_3 v^3,\]
	since $a_3 \in \mu_{q+1}$.
	\begin{remark}\label{rk:t1=0}
		Note that this is the same equation obtained for $\theta_{1}=0$, hence what follows also applies for $\theta_{1}=0$.
	\end{remark}
	\begin{lemma}\label{lm:cubes}
		Let $q=2^m$ and let $a_3$ be a cube in $\mu_{q+1}$. Then the equation $x^3 = a_3$ admits exactly $3$ solutions over $\fqs$.
	\end{lemma}
	\begin{proof}
		From \cite[pg. 4]{hirschfeld1979projective} the equation $x^3 = a_3$ has $3$ solutions if $3 \mid q^2-1$ and $a_3^\frac{q^2-1}{3}=1$. Since $q^2 \equiv 1 \pmod 3$ and $a_3^\frac{q+1}{3}=1$ the claim follows.
	\end{proof}
	
	\begin{proposition}
		Let $\cD$ be the curve with equation \eqref{eq:D}. Let $a_3$  be an element of $\mu_{q+1}$ and $\cD\colon G(u,v)=0$. Then $G(u,v)$ is irreducible over $\fqs$ if and only if $ a_3 \notin \{x^3 | x \in \mu_{q+1}\}$.
		Moreover, if $a_3 \in \{x^3 | x \in \mu_{q+1}\}$, then $\cD$ is the union of three (absolutely irreducible) linear components over $\fqs$.
	\end{proposition}
	\begin{proof}
		From Proposition \ref{pr:absirr} we know that the singular points of $\cD$ are $P_i=(a_3^q \alpha_i^2, a_3^q \alpha_i)$, for $ i=1,2,3 $, where $\alpha_i$ are the solutions in $ \overline{\Bbb F}_q $ of $x^3=a_3$.
		From Lemma \ref{lm:cubes}  $\cD$ has exactly three singular (double) points defined over $\fqs$ if and only if $a_3$ is a cube in $\mu_{q+1}$.
		Moreover, in that case, $\cD$ is the union of three (non-concurrent) lines passing through these points.
	\end{proof}
	
	\begin{corollary}\label{cor:decD}
		If $a_3$ is a cube in $\mu_{q+1}$, then $\cD$ decomposes as follows:
		\[ \cD \colon (u+\alpha_1 v + \alpha_1^{-1})(u+\alpha_2 v + \alpha_2^{-1})(u+\alpha_3 v + \alpha_3^{-1})=0 \]
	\end{corollary}
	\begin{proof}
		We just note that $\alpha_1 \alpha_2^2+\alpha_1^2 \alpha_2=a_3$. The claim follows since the line $l_i\colon u+\alpha_i v +\alpha_i^{-1}=0$ is the one passing through $P_j=(a_3^q \alpha_j ^2, a_3^q \alpha_j)$, with $j \ne i$.
	\end{proof}
	
	After that, our next goal is to understand what happens when we go back to the curve $\cC: F(X,Y)=0$, with
	\[F(X,Y)=a_3^q + a_3 X^3 Y^3 + 
	X Y (X + Y) +  (X + Y)^3 \] 
	
	\begin{proposition}\label{pr:Cfact}
		Let $a_3$ be a cube in $\mu_{q+1}$. Then the curve $\cC$ splits into linear (absolutely irreducible) components over $\fqs$. More precisely,
		\[\cC\colon \Pi_{i=1}^3 (X+\alpha_i^{-1})(Y+\alpha_i^{-1})=0, \]
		where $\alpha_i^3=a_3$ for $i=1,2,3$.
	\end{proposition}
	\begin{proof}
		The proof is a consequence of Corollary \ref{cor:decD} and $u=X+Y$, $ v=XY $. As a matter of fact, the quadric \[X+Y+\alpha_i XY + \alpha_i^{-1}=0\] splits as \[(X+\alpha_i^{-1})(Y+\alpha_i^{-1})=0\] for every $i=1,2,3$.
	\end{proof}
	
	\begin{corollary}\label{cr:a3noncube}
		Let $a_3$ be a cube in $\mu_{q+1}$. Then the set $\cC \cap \mu_{q+1}^2$ is non-empty and $f$ is not a PP of $\fqs$.
	\end{corollary}
	\begin{proof}
		The claim follows since $a_3 $ is a cube in $ \mu_{q+1}$ (and hence $\alpha_i \in \mu_{q+1}$).
	\end{proof}

	\section{$ \theta_2 \ne 0 $ and $\theta_4=0$}\label{sec:5}

	Now we suppose that $\theta_2 \ne 0$ and $\theta_4=0$. Recall that in this case \begin{equation}\label{eq:t1+t2=0}
		\theta_2^{q+1}+\theta_3^{q+1}=0.
	\end{equation}
	The equation of $\cC$ becomes
	\begin{equation}\label{eq:C2}
		\begin{aligned}
			\mathcal{C}\colon &	\theta_3^q + \theta_3 X^3 Y^3 + 
			\theta_1(X + Y)^3 
			+\theta_2 (X Y  + (X + Y)^2) + \theta_2^q (X^2 Y^2  + X Y (X + Y)^2)=0,
		\end{aligned}
	\end{equation} 
	while $\cD$ has equation
	\begin{equation*}
		\cD \colon \theta_3^q + \theta_1 u^3 + \theta_3 v^3 + \theta_2 (u^2 + v) + 
		\theta_2^q v (u^2 + v) =0.
	\end{equation*}
	Similarly to the first case, we want to understand the relation between the irreducibility of $\cD$ (and so of $\cC$).
	\begin{proposition}
		$\cC$ is absolutely irreducible if and only if $\cD$ is absolutely irreducible
	\end{proposition}
	\begin{proof}
		As we have already pointed out, the only case to be checked is when $\cC$ is the product of two cubics, which form an orbit under $\mathfrak{G}$.
		The union of two such cubics has equation $F'(X,Y)=0$, where $F'(X,Y)$ is defined as
		\begin{equation}
			\begin{aligned}
				&(a_{00} + a_{10}X + a_{20}X^2 + 
				a_{30}X^3 + a_{01}Y + a_{11}XY + 
				a_{21}X^2Y + a_{02}Y^2 + 
				a_{12}XY^2 + a_{03}Y^3)\\
				&(a_{00} + a_{01}X + a_{02}X^2 + 
				a_{03}X^3 + a_{10}Y + a_{11}XY + 
				a_{12}X^2Y + a_{20}Y^2 + 
				a_{21}XY^2 + a_{30}Y^3)=0.	
			\end{aligned}
		\end{equation} 
		By straightforward computations, we obtain 
		\begin{equation*}
			\begin{cases}
				\begin{aligned}
					&	\theta_3^q = a_{00}^2 \\
					&	 \theta_2=a_{01}^2+a_{10}^2\\
					&	 a_{00} a_{01} + a_{00} a_{10}=0
				\end{aligned}
			\end{cases}
		\end{equation*}
		Since $\theta_3^{q+1}=\theta_2^{q+1} \ne 0$, this implies $a_{00} \ne 0$ and hence $a_{10}=a_{01}$, which contradicts the assumption $\theta_2 \ne 0$.
	\end{proof}
	
	The next propositions allow us to obtain information about the factorization of $\cD$ (and so $ \cC $).
	
	\begin{proposition}\label{pr:t1=0}
		Let $\theta_1 =0$. The followings hold:
		\begin{itemize}
			\item[1.] if $\theta_3 = \theta_2^{2q-1}$ then the curve $\cD$  splits as \[\cD \colon (\theta_2+\theta_2^q v)(\theta_2^{1-q}+u^2+\theta_2^{q-1}v^2)=0;\]
			\item[2.] if $\theta_3 \ne \theta_2^{2q-1}$ then the curve $\cD$ has exactly one singular point $P=(0,\alpha)$, where $\alpha$ is the (unique) solution of $\alpha^2=\frac{\theta_2}{\theta_3}$.
		\end{itemize}
	\end{proposition}
	\begin{proof}
		The equation of $\cD$ becomes
		\[G(u,v)=u^2 v \theta_2^q+v^2 \theta_2^q+\theta_3^q+\theta_2 u^2+\theta_2 v+\theta_3 v^3=0\]
		and the partial derivatives system is made by the single equation
		\begin{equation*}
			\frac{\partial G}{\partial v}=\theta_2 + \theta_2^q u^2 + \theta_3 v^2=0
		\end{equation*}
		which implies \[u^2=\frac{\theta_3 v^2+\theta_2}{\theta_2^q}.\]
		Going back to the equation of $\cD$, we obtain
		\begin{equation}\label{eq:fvv}
			v^2 \theta_2^{2 q}+\theta_2^q \theta_3^q+\theta_2^2+\theta_2 \theta_3 v^2=0.
		\end{equation}
		Therefore, if $\theta_3 \ne \theta_2^{2q-1}$, equation \eqref{eq:fvv}, together with equation \eqref{eq:t1+t2=0}, implies \[v^2=\frac{\theta_2^{q-1} \theta_3^q+\theta_2}{\theta_2^{2 q-1}+\theta_3}=\frac{\theta_2}{\theta_3}\]
		which means that $u=0$ and $\cD$ has only one singular double point $P=(0,\alpha)$, where $\alpha^2=\frac{\theta_2}{\theta_3}$.
		
		On the other hand, if $\theta_3=\theta_2^{2q-1}$, the equation of $\cD$ becomes:
		\begin{equation}\label{eq:Dt3=t22q-1}
			v^3 \theta_2^{2 q-1}+\theta_2^{2-q}+v \theta_2^q \left(u^2+v\right)+\theta_2 \left(u^2+v\right)=0
		\end{equation}
		Note that the resultant between $h(x)$ and the derivative with respect to $v$ is $0$. This means that they share a common factor. Indeed, we have the following factorization for \eqref{eq:Dt3=t22q-1}:
		\begin{equation*}\label{key}
			\begin{aligned}
				v^3 \theta_2^{2 q-1}+\theta_2^{2-q}+v \theta_2^q \left(u^2+v\right)+\theta_2 \left(u^2+v\right)&=\\
				(\theta_2+\theta_2^q v)(\theta_2^{1-q}+u^2+\theta_2^{q-1}v^2)&=0
			\end{aligned}
		\end{equation*}
		where the second factor equals $\theta_2^{-q}\frac{\partial G}{\partial v}$.
	\end{proof}
	\begin{proposition}\label{pr:t1ne0}
		Let $\theta_1 \ne 0$. The curve $\cD$ has exactly one singular point $P=(0,\alpha)$, where $\alpha$ is the (unique) solution of $\alpha^2=\frac{\theta_2}{\theta_3}$.
	\end{proposition}
	\begin{proof}
		The system of partial derivatives is 
		\begin{equation}
			\begin{cases}
				\begin{aligned}
					&\theta_1 u^2=0\\
					&\theta_2 + \theta_2^q u^2 + \theta_3 v^2=0
				\end{aligned}
			\end{cases}
		\end{equation}
		This means that there is only one singular point $P=(0,\alpha)$ where $\alpha^2 = \frac{\theta_2}{\theta_3}$.
	\end{proof}
	Propositions \ref{pr:t1=0} and \ref{pr:t1ne0} lead us to study what kind of singular point $P=(0,\alpha)$ is. We can treat both cases together. 
	Applying a birational transformation which sends $P$ to the origin, namely
	\[\Phi: (u,v) \mapsto (U,V+\alpha),\]
	the equation for $\Phi(\cD)$ is 
	\begin{equation}\label{eq:phiC}
		(\theta_2 + \alpha \theta_2^q) U^2 + (\theta_2^q + \alpha \theta_3) V^2  + \theta_1 U^3 + \theta_2^q U^2 V +\theta_3 V^3=0.
	\end{equation}
	\begin{proposition}\label{pr:DabsIrr}
		The curve $\cD$ is absolutely irreducible if and only if $\theta_3 \ne \theta_2^{2q-1}$.
	\end{proposition}
	\begin{proof}
		The only case in which $\cD$ is absolutely irreducible is when the origin $O$ is an ordinary double point of $\Phi(\cD)$. However, when $\theta_3 = \theta_2^{2q-1}$ the equation becomes  
		\[\theta_1 U^3 + \theta_2^q U^2 V +\theta_3 V^3=0\]
		and $O$ is a triple point.
		On the other hand, when $\theta_3 \ne  \theta_2^{2q-1}$ the equation is 
		\[(U+V)(V \frac{\theta_2^q + \alpha \theta_3}{\theta_2 + \alpha \theta_2^q}+U)+\theta_1 U^3 + \theta_2^q U^2 V +\theta_3 V^3=0\]
		and $P$ is an ordinary double point. 
	\end{proof}
	
	\begin{corollary}
		Let $\theta_2 \ne 0 $ and $\theta_4=0$. If $\theta_1=0$ then $\cC \cap \mu_{q+1}^2$ is non-empty and disjoint from the line $X=Y$, whereas if $\theta_1 \ne 0$ and $\theta_2^{2q-1} \ne \theta_3$ then $\cC$ is absolutely irreducible over $\fqs$ (over $\fq$) if $\theta_2 \in \fqs \setminus \fq$ ($\theta_2 \in \fq$).
	\end{corollary}
	\begin{proof}
		The proof is obtained summing up previous propositions. More precisely, if $\theta_1=0$ and $\theta_2^{2q-1}=\theta_3$, from Proposition \ref{pr:t1=0}, we have \[\cC\colon(\theta_2+\theta_2^q XY)(\theta_2^{1-q}+(X+Y)^2+\theta_2^{q-1}X^2Y^2)=0.\]
		Let $\alpha \in \mu_{q+1} \setminus \{\frac{1}{\theta_{2}^{q-1}}\}$, then $(1/(\alpha \theta_{2}^{q-1}),\alpha) \in \cC \cap \mu_{q+1}^2$, off the line $X=Y$.
		On the other hand, if $\theta_2^{2q-1} \ne \theta_3$, the proof follows from Proposition \ref{pr:t1ne0} and \ref{pr:DabsIrr}.
	\end{proof}
	We now want to further investigate the remaining case $\theta_3=\theta_2^{2q-1}$ and $\theta_1 \ne 0$. The equation for $\Phi(\cD)$ is
	\[\theta_1 U^3 + \theta_2^q U^2 V + \theta_2^{-1 + 2 q} V^3=0.\]
	Let $Z=\frac{V}{U}$ and $z=\theta_2^q Z$. Then every solution of 
	\[\theta_1+z +\frac{1}{\theta_2^{q+1}}z^3=0\]
	gives a linear component of $\Phi(\cD)$.
	
	\begin{lemma}\label{lm:z123}
		Let $\theta_1,\theta_2 \ne 0$ and $z_1,z_2,z_3$ be the solutions of \begin{equation}\label{eq:z}
			\theta_1+z +\frac{1}{\theta_2^{q+1}}z^3=0
		\end{equation} in the algebraic closure $\overline{\Bbb F}_q$ of $\fqs$. Only one of the following conditions holds.  \begin{itemize}
			\item $z_i \in \fq$ for $i=1,2,3$.
			\item There exists $j $ such that $z_j \in \fq$ and $z_i \in \fqs$ for $i \ne j$. 
			\item $z_i \notin \fqs$ for $i=1,2,3$.
		\end{itemize} 
	\end{lemma}
	\begin{proof}
		Note that the coefficients of Equation \eqref{eq:z} are in $\fq$. The claim is obtained by standard theory, see for example \cite[Pg. 20]{hirschfeld1979projective}.
	\end{proof}
	
	\begin{proposition}
		Let $\theta_3=\theta_2^{2q-1}$. If $ \theta_1+z +\frac{1}{\theta_2^{q+1}}z^3=0 $ has at least one solution in $\fq$ then the curve $\cC$ splits as the union of three absolutely irreducible conics defined over $\fqs$ (over $\fq$) if $\theta_2 \in \fqs \setminus \fq$ ($\theta_2 \in \fq$). In particular, $\cC \cap \mu_{q+1}^2$ is a non-empty set disjoint from the line $X=Y$.
	\end{proposition}
	\begin{proof}
		Every solution of Equation \eqref{eq:z} in $\fqs$ gives a linear component of $\Phi(\cD)$ (and $\cD$). From Lemma \ref{lm:z123}, without loss of generality, we can suppose that $z_1 \in \fq$ and $z_2,z_3 \in \fqs$ are the solutions of Equation $\eqref{eq:z}$. Going back to the curve $\cD$ we obtain the following decomposition:
		\begin{equation*}
			\cD\colon ( z_1 u+ \theta_2^q v + \theta_2^q\alpha)( z_2u+ \theta_2^qv + \theta_2^q\alpha)( z_3u+\theta_2^qv +\theta_2^q\alpha)=0
		\end{equation*} 
		This means that the equation of the curve $\cC$ becomes
		\begin{equation*}
			\cC\colon (z_1 (X+Y) +\theta_2^q X Y  + \theta_2)( z_2(X+Y)+ \theta_2^qXY + \theta_2)( z_3(X+Y)+\theta_2^qXY +\theta_2)=0 
		\end{equation*}
		In fact $\alpha^2=\frac{\theta_2}{\theta_2^{2q-1}}$ implies $\alpha\theta_2^q= \frac{\theta_2^q}{\theta_2^{q-1}} =\theta_2$.
		We now claim that the above conics are absolutely irreducible over $ \fqs $.
		A conic is absolutely irreducible if and only if it does not have a singular point. Consider the conic corresponding to $z_1$, the partial derivatives system is
		\begin{equation*}
			\begin{cases}
				\begin{aligned}
					\theta_2^q Y + z_1=0\\
					\theta_2^q X + z_1=0
				\end{aligned}
			\end{cases}
		\end{equation*}
		which means that a singular point has coordinate $X=Y=\frac{z_1}{\theta_2^q}$.
		Such a point belongs to $\cC$ if and only if 
		\[z_1^2+\theta_2^{q+1}=0.\]
		However, if $z_1^2 =\theta_2^{q+1}$, from equation \eqref{eq:z} we obtain
		\[\theta_1+z_1+z_1=\theta_1=0\]
		and this is in contrast with $\theta_1 \ne 0.$
		Similarly, it can be proven that also the other conics are absolutely irreducible.
		Finally, let $\alpha \in \mu_{q+1} \setminus \{\frac{1}{\theta_{2}^{q-1}}\}$, then the point $(\frac{\theta_{2}+z_1 \alpha}{\theta_{2}^q\alpha+z_1},\alpha) \in \cC \cap \mu_{q+1}^2$, off the line $X=Y$.
	\end{proof}
	\begin{corollary}\label{cr:final}
		Let $\theta_2 \ne 0$ and $\theta_4=0$. If either $\theta_1=0 $ or $\theta_3 \neq \theta_2^{2q-1}$ or $ \theta_1+z +\frac{1}{\theta_2^{q+1}}z^3=0 $ has solutions $ z $ defined over $\fq$, then $f$ is not a PP of $\fqs$.
	\end{corollary}

	\section{$\theta_2\ne0$ and $\theta_4\ne 0$} \label{sec:6}
	We just need to prove that in this case the polynomial $f(\mathrm{X})$ is never a PP. We will do that again by using the connection between permutation polynomials and algebraic curves. This part is inspired by the work done in \cite{bartoli2018conjecture}.
	In Section 3 we showed that the polynomial $f(x)$ in Theorem \ref{th:main2} is a PP if and only if $\cH$ has no $\fq$-rational points off the line $X=Y$.
	In our case the curve $\cH$ has degree at most $6$. By Proposition \ref{pr:A-P}, for $ q $ large enough such a curve has no $\fq$-rational points off the line $X=Y$ if only if it splits into absolutely irreducible components not defined over $\fq$ which have no $\fq$-rational points off the line $X=Y$. We will show that for $\theta_{2} \ne 0$ and $\theta_4 \ne 0$ this is never the case.

	For this last section, our method requires a computer to assist us in computing resultants between polynomials and in factorizing polynomials of low degrees over small fields. The elementary MAGMA \cite{Magma} programs used for our purposes are presented in the Appendix.
	However, we point out that our results are valid for general $q$'s of type $2^m$, and do not rely on computer searches.
	
	Let $k \in \fq$ be an element of absolute trace (over $\mathbb{ F}_2$) equal to $1$. Then we can choose $i \in \fqs$ such that $i^2=i+k$ and in particular $i^q=i+1$.
	
	Let $\theta_{2}=C+iD$, $\theta_3=E+iF$, for $C,D,E,F \in \fq$.
	By direct computations the curve $\cH$ has equation $ L(X,Y)=0 $, for:
	\begin{equation}\label{EquazioneCurva}
		\begin{aligned}
			L(X,Y)=&\gamma_{3,3}X^3Y^3+\gamma_{3,2}X^3Y^2+\gamma_{2,3}X^2Y^3+\gamma_{3,1}X^3Y+\gamma_{1,3}XY^3+\gamma_{3,0}X^3+\gamma_{0,3}Y^3+\\
			+\gamma_{2,2}X^2Y^2&+\gamma_{2,1}X^2Y+\gamma_{1,2}XY^2+\gamma_{2,0}X^2+\gamma_{1,1}XY+\gamma_{0,2}Y^2+\gamma_{1,0} X+\gamma_{0,1}Y+\gamma_{0,0},
		\end{aligned}
	\end{equation}
	with 
	$$\begin{array}{lll}
		\gamma_{3,3}&=&D+F\\
		\gamma_{3,2}&=&C+D+E+F+\theta_4\\
		\gamma_{3,1}&=&C + Dk + D + E + Fk + F + \theta_4\\
		\gamma_{3,0}&=&Ck + C + Ek + E + F + k\theta_4 + \theta_4 + \theta_1\\
		\gamma_{2,3}&=&C + D + E + F + \theta_4\\
		\gamma_{1,3}&=&C + Dk + D + E + Fk + F + \theta_4\\
		\gamma_{0,3}&=&Ck + C + Ek + E + F + k\theta_4 + \theta_4 + \theta_1\\
		\gamma_{2,2}&=&C + Dk + D + E + Fk + F,\\
		\gamma_{2,1}&=&Ck + C + Ek + E + F + k\theta_4 + \theta_1,\\
		\gamma_{1,2}&=&Ck + C + Ek + E + F + k\theta_4 + \theta_1,\\
		\gamma_{2,0}&=&C + Dk^2 + Dk + E + Fk^2 + Fk + F + k\theta_4 \\
		\gamma_{0,2}&=& C + Dk^2 + Dk + E + Fk^2 + Fk + F + k\theta_4\\
		\gamma_{1,1}&=&C + Dk^2 + Dk + E + Fk^2 + Fk + F,\\
		\gamma_{1,0}&=& Ck^2 + Ck + Dk^2 + Ek^2 + Ek + E + Fk^2 + F + k^2\theta_4,\\
		\gamma_{0,1}&=&C k^2+C k+D k^2+F k^2+F+k^2 \theta_4+E k^2+E k+E,\\
		\gamma_{0,0}&=&C k^2+D k^3+F k^3+F k+F+E k^2+E.\\
	\end{array}
	$$
	
	In the following we will show that if $\theta_4 \ne 0$ and $\theta_{2}\ne 0$ then $\cH$ never splits into components none of them is defined over $\fq$.
	
	\subsection{Case $\gamma_{3,3} \ne 0$}$\\$
	In this case $\cH$ has degree $6$. We observe that the morphism $(x,y) \mapsto (y,x)$ fixes $\cH$ and therefore it acts on its components. Also, since $\cH$ is defined over $\fq$, then the Frobenius $\phi_q(x)=x^q$ acts on its components either.
	This implies that if there is a line as a component, then there must be $6$ lines. If not, $\cH$ splits as either $3$  absolutely irreducible conics or $2$ absolutely irreducible cubics.
	\begin{itemize}
		\item[1.] $\cH$ splits into 6 lines. In this case the factorization of $L(X,Y)$ in Equation \eqref{EquazioneCurva} must be
		\begin{equation}\label{eq:6l}
			(D+F)(X+a)(X+b)(X+c)(Y+a)(Y+b)(Y+c)
		\end{equation} for some $a,b,c$ in $\overline{\mathbb{F}}_q$, since the homogeneous part of degree $6$ is $(D+F)x^3y^3$.
		Now we get \begin{equation*}
			\begin{aligned}
&C + Dk + Dab + Dac + Dbc + D + E + Fk + Fab + Fac + Fbc + F + \theta_4=0\\
&C + Da + Db + Dc + D + E + Fa + F\
b + Fc + F + \theta_4=0
			\end{aligned}
		\end{equation*}
		which implies $k + ab + ac + a + bc + b + c=0$ since $D+F\ne0$.
		Last condition implies:
		\begin{equation*}
			\begin{aligned}
				&	Ca^2 + Cb^2 + C + D k^2 + D k a^2 + D k b^2 + D k + D a^4 + D a^2 b^2 + D a^2 + D b^4 + D b^2 + D +  \\
				&	E a^2+E b^2 + E + F k^2 + F k a^2 + F k b^2 + F k + F a^4 + F a^2 b^2 + F a^2 + F b^4 + F b^2 + F=0\\
				&C a^2 + C b^2 + C + D k^2 + D k a^2 + D k b^2 + D k + D a^4 + D a^2 b^2 + D b^4 + E a^2 + E b^2 + E +\\
				& F k^2 +F k a^2 + F k b^2 + F k + F a^4 + F a^2 b^2 + F a^2 + F b^4 + F b^2 + F=0
			\end{aligned}
		\end{equation*} 
		which implies in particular: $Da^2 + Db^2 + D=0$. Now let $D \ne 0$. Then $a^2 + b^2 + 1=0$, which implies $k = b^2 + b + 1$, $a=b^2+1$ and $c=b+1$.
		Thus, by direct computation, we get 
			\begin{equation*}
			\begin{aligned}
				&C + Db^4 + Db^2 + Db + D + E + Fb^4 + Fb^2 + Fb=0,\\
				&C + Db^4 + Db^2 + Db + E + Fb^4 + Fb^2 + Fb=0,
			\end{aligned}
		\end{equation*}
		which implies $D=0$ a contradiction.

		On the other hand, if $D=0$, we obtain
		\begin{equation*}
			\begin{aligned}
&C + E + Fa + Fb + Fc + F + \theta_4=0\\
&C + E + Fk + Fab + Fac + Fbc + F + \theta_4=0
			\end{aligned}
		\end{equation*}
		which implies:
		\[k + ab + ac + a + bc + b + c=0.\]
		Substituting $k$ we get
		\[C + E + Fa + Fb + Fc + F + \theta_4=0,\]
		and by eliminating $E$ we obtain
		\begin{equation*}
			\begin{aligned}
				&Fa^2 + Fab + Fac + Fb^2 + Fbc + Fc^2 + \theta_4 + \theta_1=0,\\
				&Fa^2 + Fab + Fac + Fb^2 + Fbc + Fc^2 + \theta_4=0
			\end{aligned}
		\end{equation*}
		which implies $\theta_{1}=0$. By definition we know that $\theta_{2}^{q+1}+\theta_3^{q+1}=\theta_4(\theta_4+\theta_1)$. Since $D=0$ and $\theta_{1}=0$, 
		the latter becomes $ C^2 + E^2 + EF + F^2k + \theta_4^2=0 $. By direct computations we get
		\[E + Fk + Fa^2 + Fb^2 + Fc^2 + F=0,\]
		which implies $C=0$ or $F=0$. In both cases we have a contradiction.
		\item[2.] $\cH$ splits into $3$ absolutely irreducible conics. This case is only possible when the three conics belong to the same orbit under the Frobenius $\phi_q$. More precisely, the equations of the curve must be of the form
		\begin{equation*}
			\begin{aligned}
				&(D+F)(XY +a(X+Y)+b)(XY+a^q(X+Y)+b^q)(XY+a^{q^2}(X+Y)+b^{q^2})=0
			\end{aligned}
		\end{equation*}
	for $a,b \in \mathbb{ F}_{q^3}$. First we suppose that both $a $, $b \in \mathbb{ F}_{q^3} \setminus\fq$. In this case, $ \{1,a,a^2\} $ and $\{1,b,b^2\}$ are linearly independent over $ \fq $. Also, we know that $a^3=c_1 a +c_2$ and $b^3=d_1 b +d_2$, for some $c_1,c_2,d_1,d_2 \in \fq$. First, we derive $y$ from the equation of the first conic and then we plug it in the equation of our curve. After that, we isolate the coefficients of $a$ and $a^2$. By direct computations we obtain:
	\[C + D + E + F + \theta_4=0, \mbox{ or } \theta_4=0\]
	which implies $C + D + E + F + \theta_4=0$. It follows $k=c_1$ and 
	\begin{equation}\label{eq:a}
		DE + DF + Dc_2\theta_4 + D\theta_4 + EF + F^2 + Fc_2\theta_4 + F\theta_4 + \theta_4^2=0
	\end{equation}
	Repeating for $b$ we obtain
	\begin{equation}
		\begin{aligned}\label{eqs:b}
		 &Dk + Fk + \theta_4=0\\
		&D^2 + DF + D^2d_1 + F^2d_1 + \theta_4^2=0\\
		& D^2d_2 + DE + DF + EF + F^2d_2 + F^2 + F\theta_4 + \theta_4^2=0
		\end{aligned}
	\end{equation}
	The first two equations in \eqref{eqs:b} imply that $Da^2b + D + Fa^2b + a\theta_4 + b\theta_4=0$. From the latter equation together with the third equation in \eqref{eqs:b} and Equation \eqref{eq:a} we derive that either $Da + E + F + \theta_4=0$ or $(DE + DF + D\theta_4 +EF + F^2 + F\theta_4)a+D\theta_4=0$, which leads to $D=0$ and one of the following: either $a=0$ or $E+F+\theta_4=0$ or $F=0$. However, since $D=0$ and $D+C+E+F+\theta_4=0$, the latter two conditions are both non-admissible (we would have either $D+F=0$ or $ \theta_2=0 $). Hence $a=0$, a contradiction.
	When either $a \in \fq$ or $b \in \fq$ we derive easily a contradiction by direct checking (see the Appendix for all the computations).

		\item [3.] $\cH$ splits into $2$ absolutely irreducible cubics defined over $\fqs$. The leading homogeneous part of $L(X,Y)$ is 
		$(D+F)X^3Y^3$, so the homogeneous part of the cubics is either $X^3, Y^3$ or $X^2 Y, X Y^2$. Since the Frobenius $\phi_q$ switches the two cubics, this implies that they must be defined over $\fq$.

	\end{itemize}
	
	\subsection{Case $\gamma_{3,3} = 0$.}$\\$
	If $\theta_4 \ne C+E$, $\cH$ has degree $5$. In this case we note that the line $X+Y=0$ cannot be a component of $\cH$. In fact, by direct computations, $X=Y$ implies $C+E=F=0$ and $E=0$ which in particular means $\theta_{2}=0$.
	Now, since the leading homogeneous part is $(C+E+\theta_4)(X^3Y^2+X^2Y^3)$, the point $P=(1:1:0)$ is a simple $\fq$-rational point. Then there  must be an absolutely irreducible component through $P$ distinct from the line $X+Y=0$.
	
	Let $\theta_4=C+E$. Note that $\theta_{2}^{q+1}+\theta_3^{q+1}=\theta_4 (\theta_{1}+\theta_4)$, we obtain
	\[CD + C\theta_1 + D^2k + EF + E\theta_1 + F^2k=0,\]
	and, since $D+F=0$, 
	\[F + \theta_1=0 \mbox{ or } C+E=0\]
	which implies $F=\theta_1$. 
	In this case the homogeneous part of $L(x,y)$ is $(C+E)x^2y^2$. 
	
	Since now the degree of $\cH$ is $4$ we need to deal only with two cases: $4$ lines or $2$ absolutely irreducible conics.
	\begin{itemize}
		\item[1.] $\cH$ splits as the union of $4$ lines. The factorization of $L(x,y)$ must be 
		\begin{equation}\label{eq:4l}
			(C+E)(X+a)(X+b)(Y+a)(Y+b)=0
		\end{equation}
		for some $a,b \in \overline{\mathbb{ F}}_q$. Then 
		\[a+b+1=0\]
		which implies $F=0$. It follows that
		\[k + b^2 + b + 1=0\]
		which leads to $C=0$. But this is a contradiction since $C=D=0$ implies $\theta_2=0$.
		
		\item[2.] $ \cH $ splits as the union of two absolutely irreducible conics. Since those conics are switched by $\phi_q$, we have only two possibilities, according whether they  are switched by $(x,y) \mapsto(y,x)$ or not, that is either
		$$XY+(a+i b)X+(a+(i+1)b)Y+c=0,  \textrm{ and }    XY+(a+(i+1)b) X+(a+i b) Y+c=0,$$
		for some $a,b,c \in \mathbb{F}_q$,	or 
			\[X (a+b i)+Y (a+b i)+c+d i+XY=0, \mbox{ and }\]
		\[X (a+b (i+1))+Y (a+b (i+1))+c+d (i+1)+XY=0,\]
			for some $a,b,c,d \in \mathbb{F}_q$. 
			
		In the first case we get $b+1=0$ which leads to $F=0$ and then $a^2+a+1=0$. It follows 
		\[Ck+Cc+Ek+Ec+E=0\]
		which implies either $C=0$ or $E=0$. Since $C=0$ is again a contradiction this means that $E=0$ and $c=k$. 
		However when $F=\theta_{1}=D=E=0$, the equation of our curve $\cC$ (see Equations \eqref{eq:CurveC} and \eqref{eq:C}) becomes 
		\[C(X+1)(Y+1)(X^2+XY+Y^2)=0,\]
		for $\theta_{2}=\theta_4=C$,
		which has points $(X,Y) \in \mu_{q+1}^2$ off the line $X=Y$.
		
	In the second case, by direct computation, we obtain $b=1$ and after substituting,
		\[Cd + C + Ed + E + F=0,\]
		which implies 
		\[Ck + Ca + Cc + Ek + Ea + Ec + E + Fa + F=0.\] 
		Again, by direct computation, one finds
		\[Ca^2 + Ca + C + Ea^2 + Ea + E + F=0.\]
		Now we distinguish two cases. If $F\ne0$, then
		\[k=(C^3 - C^2F - CF^2 - F^3 -  C^2E - CFE - F^2E)/(F^2(C + E)),\]
		and by replacing $k$ in the equation of our curve $\cH$, we obtain the following factorization of $L(X,Y)$:
		\[(FX+C)(FY+C)L'(X,Y), \]
		(for the equation of $L'(X,Y)$ see the Appendix) which leads to a contradiction, for the conics being irreducible.
		If $F=0$, by direct computations, we obtain $b=1$, $d=1$, $a^2+a+1=0$ and 
		\[Ck^2 + Ck + Cc^2 + Cc + Ek^2 + Ek + Ec^2 + Ec + E=0,\]
		which implies $C=0$, a contradiction again since $\theta_{2}\ne 0$.
	\end{itemize}

	\section{Proof of main Theorem \ref{th:main2}}
	If $\theta_2=0$, the proof is obtained by Corollary \ref{cr:Cabsirr} and Corollary \ref{cr:a3noncube}.
	When $\theta_2 \ne0$ and $\theta_4=0$, the proof follows from Corollary \ref{cr:final}, since the equation
	\[\theta_1+z +\frac{1}{\theta_2^{q+1}}z^3=0\]
	is equivalent to the equation \eqref{eq:trin} after substituting $z=\theta_2^{q+1}x$.	
	When $\theta_{2} \ne 0$ and $\theta_4 \ne 0$, the proof is obtained in Section \ref{sec:6}.

	\section{Appendix}
	In \cite{bartoli2018conjecture}, the author provided a very useful mini-program to compute resultant of polynomials over finite fields. In what follows we will use the same Magma procedure to investigate the solutions of a system of polynomial equations.
	For the sake of completeness, we now recall the main functions we use in this paper ``FindCoefficients2" and ``Substitution".  See \cite[Appendix]{bartoli2018conjecture} for more details.
	
	{\footnotesize
		\begin{verbatim}
		K<x,y,C,D,E,F,i,j,m,k,a,b,c,d,e,f,g,t4,t1,aq,bq,aq2,bq2> := PolynomialRing(GF(2),23);
			
			FindCoefficients2 := function(pol,var1,var2)
			T := Terms(pol);
			Coeff := {};
			MAX1 := Degree(pol,var1);
			MAX2 := Degree(pol,var2);
			for i in [0..MAX1] do
			for j in [0..MAX2] do
			c := K!0;
			for t in T do
			if IsDivisibleBy(t,var1^i*var2^j) eq true and 
			IsDivisibleBy(t,var1^i*var2^(j+1)) eq false  and 
			IsDivisibleBy(t,var1^(i+1)*var2^j) eq false  then
			c := c+ K! (t/(var1^i*var2^j));
			end if;
			end for;
			if c ne 0 then
			Coeff := Coeff join {c};
			i,j,c;
			end if;
			end for;
			end for;
			return Coeff;
			end function;
			
			Substitution := function (pol, m, p)
			e := 0;
			New := K! pol;
			while e eq 0 do
			N := K!0;
			T := Terms(New);
			i:= 0;
			for t in T do
			if IsDivisibleBy(t,m) eq true then
			Q := K! (t/m);
			i := 1;
			N := K!(N + Q* p);
			else 
			N := K!(N + t);
			end if;
			end for;
			if i eq 0 then 
			return New;
			else	
			New := K!N;
			end if;	
			end while;
			end function;
		\end{verbatim}
	}
	
	{\footnotesize
	\begin{verbatim}
		t2:=C+i*D;
		t2q:=C+(i+1)*D;
		t3:=E+i*F;
		t3q:=E+(i+1)*F;
		eq1:=Substitution(t2*t2q+t3*t3q+t4*(t4+t1),i^2,i+k);
		X:=(x+i)/(x+i+1);
		Y:=(y+i)/(y+i+1);
		Gxy:=t3q + t3*X^3*Y^3 + 
		t4*X*Y*(X + Y) + 
		(t1 + t4)*(X + Y)^3 + 
		t2*(X*Y + (X + Y)^2) + 
		t2q*X*Y*(X*Y + (X + Y)^2);
		Curve:=(x+i+1)^3*(y+i+1)^3*Gxy;
		Curve:=Substitution(Curve,i^2,i+k);
	\end{verbatim}
\subsection{$\gamma_{3,3} \ne 0$.}$ \\$
$ \cH $ splits as $6$ lines.

Case $D\ne0$.
\begin{verbatim}
		PROD := (x+a)*(x+b)*(x+c)*(y+a)*(y+b)*(y+c); 
		CC := FindCoefficients2(Curve+ (D+F)*PROD,x,y);
		{Factorization(pol) : pol in CC | pol ne 0}; 
		///C + D*a + D*b + D*c + D + E + F*a + F*b + F*c + F + t4+C + D*k + D*a*b + D*a*c + D*b*c + D + E + F*k + F*a*b + F*a*c + F*b*c + F + t4=0
		/// C + D*k + D*a*b + D*a*c + D*b*c + D + E + F*k + F*a*b + F*a*c + F*b*c + F + t4=0
		p2 := k+ a*b + a*c + a + b*c + b + c; 
		CC2 := {Resultant(pol,p2,c) : pol in CC};
		{Factorization(pol) : pol in CC2 | pol ne 0};
		///D*a^2 + D*b^2 + D=0 
		p3:=a^2 + b^2 + 1;
		CC3 := {Resultant(pol,p3,a) : pol in CC2};
		{Factorization(pol) : pol in CC3 | pol ne 0};
		///a=b^2+1, k=b^2+b+1, c=b+1
		PROD := (x+b^2+1)*(x+b)*(x+b+1)*(y+b^2+1)*(y+b)*(y+b+1);
		CC := FindCoefficients2(Curve+ (D+F)*PROD,x,y);
		CC := {Resultant(pol,k+b^2+b+1,k) : pol in CC};
		{Factorization(pol) : pol in CC | pol ne 0};
		///D=0 ###
	\end{verbatim}
Case $D=0$.
\begin{verbatim}
	PROD := (x+a)*(x+b)*(x+c)*(y+a)*(y+b)*(y+c); 
	CC := FindCoefficients2(Curve+ (D+F)*PROD,x,y);
	CC:={Resultant(pol,D,D) : pol in CC};
	{Factorization(pol) : pol in CC | pol ne 0};
	p2 := k + a*b + a*c + a + b*c + b + c; 
	CC2 := {Resultant(pol,p2,k) : pol in CC};
	{Factorization(pol) : pol in CC2 | pol ne 0};
	p3:=C + E + F*a + F*b + F*c + F + t4;
	CC3 := {Resultant(pol,p3,E) : pol in CC2}; 
	{Factorization(pol) : pol in CC3 | pol ne 0};
	///t1=0 
	eq1_2:=Substitution(eq1,t1,0);
	eq1_2:=Substitution(eq1_2,D,0);
	CC := FindCoefficients2(Curve+ (D+F)*PROD,x,y);
	CC:={Resultant(pol,D,D) : pol in CC};
	CC:={Resultant(pol,t1,t1) : pol in CC};
	p1:=eq1;
	CC1:={Resultant(pol,p1,C) : pol in CC};
	{Factorization(pol) : pol in CC1 | pol ne 0};
	p2 :=E + F*k + F*a^2 + F*b^2 + F*c^2 + F; 
	CC2 := {Resultant(pol,p2,a) : pol in CC};
	{Factorization(pol) : pol in CC2 | pol ne 0};
	/// C=0 or F=0 ###
\end{verbatim}
$ \cH $ splits as $3$ absolutely irreducible conics.
Case $\alpha,\beta \in \mathbb{ F}_{q^3} \setminus \fq$.
\begin{verbatim}
Curve1:=K!((x+a)^3*Evaluate(Curve, [x,(b+a*x)/(x+(a)),C,D,E,F,i,j,m,k,
a,b,c,d,e,f,g,t4,t1,aq,bq,aq2,bq2]));
CC := FindCoefficients2(Curve1,x,y);
CC:={Resultant(pol,eq1,t1) : pol in CC};
CC:={Substitution(pol,a^3,m*a+g) : pol in CC};
CC:={Substitution(pol,b^3,i*b+j) : pol in CC};
{Factorization(pol) : pol in CC | pol ne 0};
pa1:=C^2 + C*D + C*k*t4 + C*a^2*t4 + C*a*t4 + C*t4 + D^2*k + D*m*a*t4 + D*k*a*t4 +
 D*a^2*t4 + D*a*t4 + D*g*t4 + E^2 + E*F + E*k*t4 + E*a^2*t4 + E*a*t4 + E*t4 +
 F^2*k + F*m*a*t4 + F*k*a*t4 + F*a^2*t4 +  F*a*t4 + F*g*t4 + F*t4 + k*t4^2 + a^2*t4^2 + a*t4^2;
Coefficients(pa1,a);
p1:=C + D + E + F + t4;
CC1:={Resultant(pol,p1,C) : pol in CC};
{Factorization(pol) : pol in CC1 | pol ne 0}; 
 pa2:=D*k + E + F*k + F + m*a*t4 + k*a*t4 + k*t4 + g*t4;
Coefficients(pa2,a);
p2:=k +m;
CC2 := {Resultant(pol,p2,m) : pol in CC1}; 
{Factorization(pol) : pol in CC2 | pol ne 0};
pb1:=D*i*b + D*j + D*k^3 + D*k^2*b + D*k^2 + D*k*b^2 + D*k*b + D*b + E + F*i*b + F*j +
 F*k^3 + F*k^2*b + F*k^2 + F*k*b^2 + F*k*b + F*k + F + k^2*t4 + b^2*t4 + b*t4;
Coefficients(pb1,b);
p3:= D*k + F*k + t4;
CC3:={Resultant(pol,p3,k) : pol in CC2}; 
{Factorization(pol) : pol in CC3 | pol ne 0};
pb2:=D^2*i*b + D^2*j + D^2*b + D*E + D*F*b + D*F + E*F + F^2*i*b +
 F^2*j + F^2 + F*t4 + b*t4^2 + t4^2;
Coefficients(pb2,b);
p4:= D^2 + D*F + D^2*i + F^2*i + t4^2;
CC4:={Resultant(pol,p4,i) : pol in CC3}; 
{Factorization(pol) : pol in CC4 | pol ne 0};
p5:=D*a^2*b + D + F*a^2*b + a*t4 + b*t4;
CC5:={Resultant(pol,p5,b) : pol in CC4}; 
{Factorization(pol) : pol in CC5 | pol ne 0};
CC6:={Substitution(pol,a^3,k*a+g) : pol in CC5};
CC6:={Resultant(pol,D*k+ F*k + t4,k) : pol in CC6}; 
{Factorization(pol) : pol in CC6 | pol ne 0};
p7:=D*E + D*F + D*g*t4 + D*t4 + E*F + F^2 + F*g*t4 + F*t4 + t4^2;
CC7:={Resultant(pol,p7,g) : pol in CC6}; 
{Factorization(pol) : pol in CC7 | pol ne 0};
p8:=D^2*j + D*E + D*F + E*F + F^2*j + F^2 + F*t4 + t4^2;
CC8:={Resultant(pol,p8,j) : pol in CC7}; 
{Factorization(pol) : pol in CC8 | pol ne 0};
p9:=D;
CC9:={Resultant(pol,p9,D) : pol in CC8}; 
{Factorization(pol) : pol in CC9 | pol ne 0};
///a=0 ###
\end{verbatim}
Case $\alpha \in \fq$ or $\beta \in \fq$.
\begin{verbatim}
	PROD1:= (x*y+a*(x+y)+b)*(x*y+a*(x+y)+bq)*(x*y+a*(x+y)+bq2);
	PROD2:= (x*y+a*(x+y)+b)*(x*y+aq*(x+y)+b)*(x*y+aq2*(x+y)+b);
	CC := FindCoefficients2(Curve+(D+F)*PROD1,x,y);
	CC:={Resultant(pol,eq1,t1) : pol in CC};
	{Factorization(pol) : pol in CC | pol ne 0};
	/// t4=0 (by sum of the two equations starting by C^2+CD+..) ###
	CC := FindCoefficients2(Curve+(D+F)*PROD2,x,y);
	CC:={Resultant(pol,eq1,t1) : pol in CC};
	{Factorization(pol) : pol in CC | pol ne 0};
	/// t4=0 (by sum of the two equations starting by C^2+CD+..) ###
\end{verbatim}

\subsection{$\gamma_{3,3}=0$.}$ \\$
In this case we recall that $D+F=0$ and $\theta_4=C+E$, implying $\theta_{1}=F$.
\begin{verbatim}
	Curve2:=Substitution(Curve,D,F);
	Curve2:=Substitution(Curve2,t4,C+E);
	Curve2:=Substitution(Curve2,t1,F);
\end{verbatim}
$ \cH $ splits as $4$ lines.
\begin{verbatim}
	PROD := (x+a)*(x+b)*(y+a)*(y+b); 
	CC := FindCoefficients2(Curve2+ (C+E)*PROD,x,y);
	{Factorization(pol) : pol in CC | pol ne 0};
	p2 := a + b + 1; 
	CC2 := {Resultant(pol,p2,a) : pol in CC};
	{Factorization(pol) : pol in CC2 | pol ne 0};
	p3:=F;
	CC3 := {Resultant(pol,p3,F) : pol in CC2};
	{Factorization(pol) : pol in CC3 | pol ne 0};
	p4:=k + b^2 + b + 1;
	CC4 := {Resultant(pol,p4,k) : pol in CC3}; 
	{Factorization(pol) : pol in CC4 | pol ne 0};
	/// C=0 ###
\end{verbatim}
$\cH$ splits as the union of two absolutely irreducible conics.

Case $1$: $(x,y)\mapsto(y,x)$ switches the two conics.
\begin{verbatim}
	PROD := (x*y+(a+i*b)*x+(a+(i+1)*b)*y+c)*(x*y+(a+(i+1)*b)*x+(a+i*b)*y+c); 
	PROD := Substitution(PROD,i^2,i+k);
	CC := FindCoefficients2(Curve2+(C+E)*PROD,x,y);
	{Factorization(pol) : pol in CC | pol ne 0};
	p2 :=  b + 1; 
	CC2 := {Resultant(pol,p2,b) : pol in CC};
	{Factorization(pol) : pol in CC2 | pol ne 0};
	p3:=F;
	CC3 := {Resultant(pol,p3,F) : pol in CC2};
	{Factorization(pol) : pol in CC3 | pol ne 0};
	p4:=a^2+a+1;
	CC4 := {Resultant(pol,p4,a) : pol in CC3}; 
	{Factorization(pol) : pol in CC4 | pol ne 0};
	p5:=C*k + C*c + E*k + E*c + E;
	CC5:={Resultant(pol,p5,k) : pol in CC4}; 
	{Factorization(pol) : pol in CC5 | pol ne 0};
	/// E=0 ###
\end{verbatim}
Case $2$:  $(x,y)\mapsto(y,x)$ fixes the two conics.

Case $F\ne 0$.
\begin{verbatim}
	PROD := (x*y+(a+i*b)*x+(a+i*b)*y+(c+i*d))*(x*y+(a+(i+1)*b)*x+(a+(i+1)*b)*y+(c+(i+1)*d)); 
	PROD := Substitution(PROD,i^2,i+k);
	CC := FindCoefficients2(Curve2+(C+E)*PROD,x,y); 
	{Factorization(pol) : pol in CC | pol ne 0};
	p2 :=  b + 1; 
	CC2 := {Resultant(pol,p2,b) : pol in CC};
	{Factorization(pol) : pol in CC2 | pol ne 0};
	p3:=C*d + C + E*d + E + F;
	CC3 := {Resultant(pol,p3,d) : pol in CC2};
	{Factorization(pol) : pol in CC3 | pol ne 0};
	p4:=C*k + C*a + C*c + E*k + E*a + E*c + E + F*a + F;
	CC4 := {Resultant(pol,p4,c) : pol in CC3}; 
	{Factorization(pol) : pol in CC4 | pol ne 0};
	p5:=C*a^2 + C*a + C + E*a^2 + E*a + E + F;
	CC5:={Resultant(pol,p5,a) : pol in CC4}; 
	{Factorization(pol) : pol in CC5 | pol ne 0};
	/// k:=-(C^3 - C^2*F - C*F^2 - F^3 -  C^2*E - C*F*E - F^2*E)/(F^2*(C + E));
	Factorization(K!(F^4*Evaluate(Curve2,
	[x,y,C,D,E,F,i,j,m,(C^3 + C^2*F + C*F^2 + F^3 +  C^2*E + C*F*E + F^2*E)/(F^2*(C + E)),
	a,b,c,d,e,f,g,t4,t1,aq,bq,aq2,bq2])));
\end{verbatim}
Case $F=0$
\begin{verbatim}
	Curve3:=Substitution(Curve2,F,0);
	PROD := (x*y+(a+i*b)*x+(a+i*b)*y+(c+i*d))*(x*y+(a+(i+1)*b)*x+(a+(i+1)*b)*y+(c+(i+1)*d)); 
	PROD := Substitution(PROD,i^2,i+k);
	CC := FindCoefficients2(Curve3+(C+E)*PROD,x,y); 
	{Factorization(pol) : pol in CC | pol ne 0};
	p2 :=  b + 1; 
	CC2 := {Resultant(pol,p2,b) : pol in CC};
	{Factorization(pol) : pol in CC2 | pol ne 0};
	p3:=d+1;
	CC3 := {Resultant(pol,p3,d) : pol in CC2};
	{Factorization(pol) : pol in CC3 | pol ne 0};
	p4:=a^2+a+1;
	CC4 := {Resultant(pol,p4,a) : pol in CC3}; 
	{Factorization(pol) : pol in CC4 | pol ne 0};
	p5:=C*k^2 + C*k + C*c^2 + C*c + E*k^2 + E*k + E*c^2 + E*c + E;
	CC5:={Resultant(pol,p5,c) : pol in CC4}; 
	{Factorization(pol) : pol in CC5 | pol ne 0};
	/// C=0 ###
\end{verbatim}
	
	\section*{Acknowledgment}
	This work was supported by the National Science Foundation under Grant N° 2127742 
	
	
	

	
	\bibliographystyle{acm}
	\bibliography{bib_tesi}
	
\end{document}